\documentclass[10pt,a4paper,twoside]{amsart}

\usepackage{graphicx}
\usepackage{a4wide}

\newtheorem{definition}{Definition}[section]
\newtheorem{theorem}{Theorem}[section]
\newtheorem{lemma}{Lemma}[section]

\newtheorem*{maintheorem*}{Main Theorem}
\allowdisplaybreaks
\numberwithin{equation}{section}

\newcommand{\norm}[1]{\left\| #1 \right\|}

\newcommand{\eps}{\varepsilon}

\newcommand{\eb}{{\eps,\beta}}
\newcommand{\ueb}{u_\eb}

\newcommand{\pt}{\partial_t}

\newcommand{\px}{\partial_x }
\newcommand{\pxx}{\partial_{xx}^2}
\newcommand{\pxxx}{\partial_{xxx}^3}

\renewcommand{\i}{\ifmmode\mathit{\mathchar"7010 }\else\char"10 \fi}
\renewcommand{\j}{\ifmmode\mathit{\mathchar"7011 }\else\char"11 \fi}
\newcommand{\R}{\mathbb{R}}
\newcommand{\N}{\mathbb{N}}

\newcommand{\supp}{\mathrm{supp}\,}

{%

\begin{enumerate}}%
{\end{enumerate}}

{%

\begin{enumerate}}%
{\end{enumerate}}

\begin{document}\large

\title[Kudryashov-Sinelshchikov equation]{A singular limit problem for  the Kudryashov-Sinelshchikov equation}

\author[G. M. Coclite and L. di Ruvo]{Giuseppe Maria Coclite and Lorenzo di Ruvo}
\address[Giuseppe Maria Coclite and Lorenzo di Ruvo]
{\newline Department of Mathematics,   University of Bari, via E. Orabona 4, 70125 Bari,   Italy}
\email[]{giuseppemaria.coclite@uniba.it, lorenzo.diruvo@uniba.it}
\urladdr{http://www.dm.uniba.it/Members/coclitegm/}

\date{\today}



%35G25 Initial value problems for nonlinear higher-order PDE, nonlinear evolution equations
%35L05 Wave equation
%74S20 Finite difference methods
%35L65 Conservation laws
%65M12 Stability and convergence of numerical methods

\thanks{The authors are members of the Gruppo Nazionale per l'Analisi Matematica, la Probabilit\`a e le loro Applicazioni (GNAMPA) of the Istituto Nazionale di Alta Matematica (INdAM)}

\begin{abstract}
We consider the Kudryashov-Sinelshchikov equation, which contains nonlinear dispersive effects. We prove
that as the diffusion parameter tends to zero, the solutions of the dispersive equation converge to the entropy ones of the Burgers equation.
The proof relies on deriving suitable a priori estimates together with an application of the compensated compactness method in the $L^p$ setting.
\end{abstract}

\maketitle

%\tableofcontents

\section{Introduction}
\label{sec:intro}
A mixture of liquid and gas bubbles of the same size may be considered as an example of a classic nonlinear medium. The analysis of propagation of the pressure waves in a liquid with gas bubbles is an important problem. Indeed, there are solitary and periodic waves in such mixtures and they can be described by nonlinear partial differential equations like the 
Burgers, Korteweg-de Vries, and  the Burgers-Korteweg-de Vries ones.

Recently, Kudryashov and Sinelshchikov \cite{KS} obtained a more general nonlinear partial differential equation to describe the pressure waves in a liquid and gas bubbles mixture taking into consideration the viscosity of liquid and the heat transfer. They introduced the equation
\begin{equation}
\label{eq:KS}
\pt u+ Au\px u+\beta\pxxx u -B\beta \px\left(u\pxx u\right)-C\beta\px u\pxx u -\eps\pxx u - D\beta\px\left(u\px u\right)=0,
\end{equation}
where $u$ is a density and  models heat transfer and viscosity, while $A,\,\beta,\,B,\,C,\eps,\, D$ are real parameters.
If $B=C=D=0$, \eqref{eq:KS} reads
\begin{equation}
\label{eq:B-K-D}
\pt u+ Au\px u+\beta\pxxx u-\eps\pxx u=0,
\end{equation}
which is known as Korteweg-de Vries-Burgers equation \cite{Su}. If also $\eps=0$, we obtain the Korteweg-de Vries equation \cite{KdV}.

Several results have been obtained in the case
\begin{equation*} 
A=1,\quad \beta=1,\quad B=1,\quad \eps=0, \quad D=0,
\end{equation*}
in which \eqref{eq:KS} reads
\begin{equation}
\label{eq:KS-bis}
\pt u+ u\px u+\pxxx u -\px\left(u\pxx u\right)-C\px u\pxx u=0.
\end{equation}
In \cite{Ry}, the author found four families of solitary wave solutions of \eqref{eq:KS-bis} when $C=-3$, or $C=-4$. In \cite{LC}, the authors discussed the existence of different kinds of traveling wave solutions by using the approach of dynamical systems, according to different phase orbits of the traveling system
of \eqref{eq:KS-bis}; twenty-six kinds of exact traveling wave solutions are obtained under the parameter chioces $C=-3,\,-4,\,1,\,2$. In \cite{HML},  the authors discussed the bifurcations of phase portraits and investigated exact traveling wave solutions of \eqref{eq:KS-bis} in the cases $C=-3,\,1,\,2$. In \cite{HMZL}, the authors investigated periodic loop solutions of \eqref{eq:KS-bis}, and discussed the limit forms of these solutions focusing on the case $C=2$. In \cite{RA}, the author studied \eqref{eq:KS-bis} under the transformation $\alpha=2+C$, in the cases $\alpha<0$, $\alpha=0$, and $\alpha>0$ (i.e., $C<-2$, $C=-2$, and $C>-2$). He obtained some exact traveling wave solutions and discussed their dynamical behaviors. Some interesting phenomena of traveling waves are successfully explained. Particularly, when $\alpha>2$ (i.e., $C >0$), a kind of new periodic wave solutions which is called {\em meandering solutions} was obtained. In \cite{CRL}, the authors studied \eqref{eq:KS-bis} by using the integral bifurcation method (see \cite{RHLC,WYZ}). They found some new traveling wave solutions of \eqref{eq:KS-bis}, which extends the results in \cite{HML,HMZL,LC,RA,Ry}.

In this paper, we are interested to energy preserving waves, therefore we analyze \eqref{eq:KS} in the case
\begin{equation}
\label{eq:cond-cost}
\left (B,C\right)= \left(\frac{2A}{3},-\frac{A}{3}\right),\quad D=0.
\end{equation}
%The assumption \eqref{eq:cond-cost} guarantees the energy conservation for \eqref{eq:KS} (see Lemma \ref{lm:l-2}).\\
We study the Cauchy problem
\begin{equation}
\label{eq:KS1}
\begin{cases}
\pt u+ Au\px u + \beta\pxxx u -B\beta\px\left(u\pxx u\right)\\
\qquad-C\beta\px u\pxx u -\eps\pxx u =0,&\qquad t>0, \ x\in\R ,\\
u(0,x)=u_0(x), &\qquad x\in\R.
\end{cases}
\end{equation}
We study the no high frequency limit,  namely we send $\beta,\,\eps\to 0$ in \eqref{eq:KS1}. In this way we pass from \eqref{eq:KS1} to the Burgers equation
\begin{equation}
\label{eq:BU}
\begin{cases}
\pt u+Au\px u=0,&\qquad t>0, \quad x\in\R,\\
u(0,x)=u_0(x), &\qquad x\in\R.
\end{cases}
\end{equation}
On the initial datum, we assume that
\begin{equation}
\label{eq:assinit}
u_0\in L^2(\R)\cap L^4(\R).
\end{equation}
We study the dispersion-diffusion limit for \eqref{eq:KS1}. Therefore, we consider the following third order approximation
\begin{equation}
\label{eq:KS-eps-beta}
\begin{cases}
\pt\ueb+ A \ueb\px \ueb + \beta\pxxx \ueb -B\beta \px\left(\ueb\pxx \ueb\right)\\
\qquad\quad -C\beta\px \ueb\pxx \ueb -\eps\pxx \ueb =0,&\qquad t>0, \ x\in\R ,\\
\ueb(0,x)=u_{\eps,\beta,0}(x), &\qquad x\in\R,
\end{cases}
\end{equation}
where $u_{\eps,\beta,0}$ is a $C^\infty$ approximation of $u_{0}$ such that
\begin{equation}
\begin{split}
\label{eq:u0eps}
&u_{\eps,\,\beta,\,0} \to u_{0} \quad  \textrm{in $L^{p}_{loc}(\R)$, $1\le p < 4$, as $\eps,\,\beta \to 0$,}\\
&\norm{u_{\eps,\beta, 0}}^2_{L^2(\R)}+\norm{u_{\eps,\beta, 0}}^4_{L^4(\R)}+(\beta+ \eps^2) \norm{\px u_{\eps,\beta,0}}^2_{L^2(\R)}\le C_{0}, \quad \eps,\beta >0,
\end{split}
\end{equation}
and $C_0$ is a constant independent on $\eps$ and $\beta$.

The main result of this paper is the following theorem.
\begin{theorem}
\label{th:main}
Assume that \eqref{eq:assinit} and  \eqref{eq:u0eps} hold. If
\begin{equation}
\label{eq:beta-eps}
\beta=\mathbf{\mathcal{O}}(\eps^{4}),
\end{equation}
then, there exist two sequences $\{\eps_{n}\}_{n\in\N}$, $\{\beta_{n}\}_{n\in\N}$, with $\eps_n, \beta_n \to 0$, and a limit function
\begin{equation*}
u\in L^{\infty}(\R^{+}; L^2(\R)\cap\L^{4}(\R)),
\end{equation*}
such that
\begin{align}
\label{eq:con-u}
&u_{\eps_n, \beta_n}\to u \quad \text{strongly in $L^{p}_{loc}(\R^{+}\times\R)$, for each $1\le p <4$},\\
\label{eq:entropy1}
&u \quad\text{is the unique entropy solution of \eqref{eq:BU}}.
\end{align}
\end{theorem}
The paper is organized in four sections. In Section \ref{sec:vv}, we prove some a priori estimates, while in Section \ref{sec:theor} we prove Theorem \ref{th:main}. In Appendix, we prove that Theorem \ref{th:main} holds also in the case $A=\left(C+\alpha\right)^{2n}$, where $\alpha$ is a suitable real number.

\section{A priori Estimates}
\label{sec:vv}

This section is devoted to some a priori estimates on $\ueb$. We denote with $C_0$ the constants which depend only on the initial data, and with $C(T)$ the constants which depend also on $T$.
\begin{lemma}\label{lm:l-2}
Assume \eqref{eq:cond-cost}. For each $t>0$,
\begin{equation}
\label{eq:l-2}
\begin{split}
\norm{\ueb(t,\cdot)}^2_{L^2(\R)}&+\beta\norm{\px\ueb(t,\cdot)}^2_{L^2(\R)}+2\eps\int_{0}^{t}\norm{\px\ueb(s,\cdot)}^2_{L^2(\R)}ds\\
&+2\beta\eps\int_{0}^{t}\norm{\pxx\ueb(s,\cdot)}^2_{L^2(\R)}ds\le C_{0}.
\end{split}
\end{equation}
In particular, we have
\begin{equation}
\label{eq:u-l-infty}
\norm{\ueb(t,\cdot)}_{L^{\infty}(\R)}\le C_{0}\beta^{-\frac{1}{4}}.
\end{equation}
\end{lemma}
\begin{proof}
Multiplying \eqref{eq:KS-eps-beta} by $\ueb-\beta\pxx\ueb$, we have
\begin{equation}
\label{eq:ks4}
\begin{split}
&(\ueb -\beta\pxx\ueb)\pt\ueb + A(\ueb -\beta\pxx\ueb) \ueb\px \ueb\\
 &\qquad+ (\ueb -\beta\pxx\ueb)\beta\pxxx \ueb-B\beta(\ueb -\beta\pxx\ueb)\px\left(\ueb\pxx \ueb\right)\\
 &\qquad-C\beta(\ueb -\beta\pxx\ueb)\px \ueb\pxx \ueb-\eps(\ueb -\beta\pxx\ueb)\pxx \ueb =0.
\end{split}
\end{equation}
Since
\begin{align*}
\int_{\R}&\left(\ueb-\beta\pxx\ueb\right)\pt\ueb dx\\
 =&\frac{1}{2}\frac{d}{dt}\left(\norm{\ueb(t,\cdot)}^2_{L^2(\R)} + \beta\norm{\px\ueb(t,\cdot)}^2_{L^2(\R)}\right),\\
A\int_{\R}& (\ueb -\beta\pxx\ueb) \ueb\px \ueb dx = -A\beta \int_{\R}\ueb\px \ueb \pxx\ueb dx,\\
\beta \int_{\R}&(\ueb-\beta\pxx\ueb)\beta\pxxx \ueb dx = -\beta \int_{\R}\px\ueb\pxx\ueb =0,\\
-B\beta\int_{\R}& (\ueb -\beta\pxx\ueb)\px\left(\ueb\pxx \ueb\right) dx\\
 =& B\beta \int_{\R}\ueb\px \ueb \pxx\ueb dx+B\beta^2 \int_{\R} \pxx\ueb \px\left(\ueb\pxx \ueb\right) dx,\\
-C\beta\int_{\R}&(\ueb -\beta\pxx\ueb)\px \ueb\pxx \ueb dx\\
 =& -C\beta \int_{\R}\ueb\px \ueb \pxx\ueb dx + C\beta^2 \int_{\R} \px\ueb(\pxx\ueb)^2 dx,\\
-\eps\int_{\R}&(\ueb -\beta\pxx\ueb)\pxx \ueb dx= \eps\norm{\px\ueb(t,\cdot)}^2_{L^2(\R)}+ \beta\eps\norm{\pxx\ueb(t,\cdot)}^2_{L^2(\R)},
\end{align*}
integrating \eqref{eq:ks4} on $\R$, we get
\begin{equation}
\label{eq:ks6}
\begin{split}
&\frac{d}{dt}\left(\norm{\ueb(t,\cdot)}^2_{L^2(\R)} + \beta\norm{\px\ueb(t,\cdot)}^2_{L^2(\R)}\right)+2\eps\norm{\px\ueb(t,\cdot)}^2_{L^2(\R)}\\
&\qquad\quad +2\beta\eps\norm{\pxx\ueb(t,\cdot)}^2_{L^2(\R)} -2\beta\left(A-B+C\right)\int_{\R}\ueb\px \ueb \pxx\ueb dx\\
&\qquad\quad +2B\beta^2 \int_{\R} \pxx\ueb \px\left(\ueb\pxx \ueb\right) dx+ 2C\beta^2 \int_{\R} \px\ueb(\pxx\ueb)^2 dx=0.
\end{split}
\end{equation}
Observe that
\begin{align*}
2B\beta^2\int_{\R} \pxx\ueb \px\left(\ueb\pxx \ueb\right) dx=&-2B\beta^2 \int_{\R}\ueb\pxxx\ueb\pxx\ueb dx\\
=&-B\beta^2\int_{\R}\ueb\px(\pxx\ueb)^2dx \\
=&B\beta^2 \int_{\R}\px\ueb(\pxx\ueb)^2dx.
\end{align*}
Thus, from \eqref{eq:ks6},
\begin{align*}
&\frac{d}{dt}\left(\norm{\ueb(t,\cdot)}^2_{L^2(\R)} + \beta\norm{\px\ueb(t,\cdot)}^2_{L^2(\R)}\right)+2\eps\norm{\px\ueb(t,\cdot)}^2_{L^2(\R)}\\
&\qquad\quad +2\beta\eps\norm{\pxx\ueb(t,\cdot)}^2_{L^2(\R)} -2\beta\left(A-B+C\right)\int_{\R}\ueb\px \ueb \pxx\ueb dx\\
&\qquad\quad +\beta^2\left(B+2C\right)\int_{\R} \px\ueb(\pxx\ueb)^2 dx=0.
\end{align*}
Thanks to \eqref{eq:cond-cost}, we have
\begin{equation}
\label{eq:syst1}
\begin{cases}
A-B+C=0,\\
B+2C=0.
\end{cases}
\end{equation}
Therefore, \eqref{eq:l-2} follows from \eqref{eq:cond-cost}, \eqref{eq:u0eps} and an integration on $(0,t)$.

Finally, we prove \eqref{eq:u-l-infty}. Due to \eqref{eq:l-2} and the H\"older inequality,
\begin{align*}
\ueb^2(t,x)=&2\int_{-\infty}^x \ueb\px\ueb dx \le 2\int_{\R}\vert \ueb\vert \vert \px\ueb\vert dx\\
\le&2\norm{\ueb(t,\cdot)}_{L^2(\R)}\norm{\px\ueb(t,\cdot)}_{L^2(\R)}\le C_{0}\beta^{-\frac{1}{2}}.
\end{align*}
Therefore,
\begin{equation*}
\vert \ueb(t,x)\vert \le C_{0}\beta^{-\frac{1}{4}},
\end{equation*}
which gives \eqref{eq:u-l-infty}.
\end{proof}
Following \cite[Lemma $2.2$]{CdREM}, or \cite[Lemma $4.2$]{CK}, we prove the following result.
\begin{lemma}\label{lm:ux-l-2}
Assume that \eqref{eq:cond-cost} and \eqref{eq:beta-eps} hold. Then, for each $t>0$,
\begin{itemize}
\item[$i)$] the family $\{\ueb\}_{\eps,\beta}$ is bounded in $L^{\infty}(\R^{+};L^{4}(\R))$;
\item[$ii)$]the family $\{\eps\px\ueb\}_{\eps,\beta}$ is bounded in $L^{\infty}(\R^{+};L^2(\R))$;
\item[$iii)$] the families $\{\sqrt{\eps}\ueb\px\ueb\}_{\eps,\beta}, \, \{\eps\sqrt{\eps}\pxx\ueb \}_{\eps,\beta} $ are bounded in $L^2(\R^{+}\times\R)$.
\end{itemize}
Moreover,
\begin{align}
\label{eq:ux-uxx}
\beta\int_0^t\norm{\px\ueb(s,\cdot)\pxx\ueb(s,\cdot)}_{L^1(\R)}ds\le& C_0\eps^2, \quad t>0,\\
\label{eq:uxx-l-2}
\beta^2\int_{0}^t\norm{\pxx\ueb(t,\cdot)}^2_{L^2(\R)}ds \le &C_0\eps^5, \quad t>0,\\
\label{eq:u-uxx-1}
\beta^2\int_{0}^{t}\norm{\ueb(s,\cdot)\pxx\ueb(s,\cdot)}^2_{L^2(\R)}ds \le & C_0\eps^3, \quad t>0,\\
\label{eq:u-ux-uxx}
\beta\int_{0}^{t}\norm{\ueb(s,\cdot)\px\ueb(s,\cdot)\pxx\ueb(s,\cdot)}_{L^1(\R)} ds\le & C_0\eps, \quad t>0.
\end{align}
\end{lemma}
\begin{proof}
Let $K$ be a positive constant which will be specified later. Multiplying  \eqref{eq:KS-eps-beta} by $\displaystyle K\ueb^3 -\eps^2\pxx\ueb$,  we have
\begin{equation}
\label{eq:KSmp}
\begin{split}
\left(K\ueb^3 -\eps^2\pxx\ueb\right)\pt\ueb &+A\left(K\ueb^3 -\eps^2\pxx\ueb\right)\ueb\px\ueb\\
&+\beta\left(K\ueb^3 -\eps^2\pxx\ueb\right)\pxxx\ueb\\
&-B\beta\left(K\ueb^3 -\eps^2\pxx\ueb\right)\px\left(u\pxx u\right)\\
&-C\beta \left( K\ueb^3 -\eps^2\pxx\ueb\right)\px\ueb\pxx\ueb\\
&-\eps\left(K\ueb^3 -\eps^2\pxx\ueb\right)\pxx\ueb=0.
\end{split}
\end{equation}
Observe that
\begin{align*}
\int_{\R}&\left(K\ueb^3 -\eps^2\pxx\ueb\right)\pt\ueb dx=\frac{d}{dt}\left(\frac{K}{4}\norm{\ueb(t,\cdot)}^4_{L^{4}(\R)} +\frac{\eps^2}{2}\norm{\px\ueb(t,\cdot)}^2_{L^2(\R)}\right),\\
A\int_{\R}&\left(K\ueb^3 -\eps^2\pxx\ueb\right)\ueb\px\ueb= -A\eps^2\int_{\R}\ueb\px\ueb\pxx\ueb dx,\\
\beta\int_{\R}&\left(K\ueb^3 -\eps^2\pxx\ueb\right)\pxxx\ueb dx =  -3K\beta\int_{\R}\ueb^2\px\ueb\pxx\ueb dx,\\
-B\beta\int_{\R}&\left(K\ueb^3 -\eps^2\pxx\ueb\right)\px\left(u\pxx u\right)dx\\
=& 3BK\beta\int_{\R}\ueb^3\px\ueb\pxx\ueb dx+\eps^2\beta B\int_{\R}\ueb\pxx\ueb\pxxx\ueb dx\\
=& 3BK\beta\int_{\R}\ueb^3\px\ueb\pxx\ueb dx+\frac{\eps^2\beta B}{2} \int_{\R} \ueb(\pxx\ueb)^2 dx,\\
-C\beta \int_{\R}&\left( K\ueb^3 -\eps^2\pxx\ueb\right)\px\ueb\pxx\ueb dx \\
=& -CK\beta\int_{\R}\ueb^3\px\ueb\pxx\ueb dx+\eps^2\beta C  \int_{\R} \ueb(\pxx\ueb)^2 dx,\\
-\eps\int_{\R}&\left(K\ueb^3 -\eps^2\pxx\ueb\right)\pxx\ueb dx\\
=&3K\eps\norm{\ueb(t,\cdot)\px\ueb(t,\cdot)}^2_{L^2(\R)}+\eps^3\norm{\pxx\ueb(t,\cdot)}^2_{L^2(\R)}.
\end{align*}
Therefore, integrating \eqref{eq:KSmp} over $\R$, from \eqref{eq:cond-cost}, we get
\begin{equation}
\label{eq:p12}
\begin{split}
&\frac{d}{dt}\left(\frac{K}{4}\norm{\ueb(t,\cdot)}^4_{L^{4}(\R)} +\frac{\eps^2}{2}\norm{\px\ueb(t,\cdot)}^2_{L^2(\R)}\right)\\
&\qquad\quad +3K\eps\norm{\ueb(t,\cdot)\px\ueb(t,\cdot)}^2_{L^2(\R)} +\eps^3\norm{\pxx\ueb(t,\cdot)}^2_{L^2(\R)}\\
&\qquad= -A\eps^2\int_{\R}\ueb\px\ueb\pxx\ueb dx + 3K\beta\int_{\R}\ueb^2\px\ueb\pxx\ueb dx\\
&\qquad\quad +\frac{7A}{3}K\beta\int_{\R}\ueb^3\px\ueb\pxx\ueb dx\\
&\qquad\le \eps^2\left\vert A\right\vert\int_{\R}\vert\ueb\px\ueb\vert\vert\pxx\ueb\vert dx + 3K\beta \int_{\R}\ueb^2\vert\px\ueb\vert\vert\pxx\ueb\vert dx\\
&\qquad\quad +\frac{7}{3}K\beta\int_{\R}\left\vert A \ueb^3\px\ueb\right\vert \vert \pxx\ueb\vert dx.
\end{split}
\end{equation}
Due to the Young inequality,
\begin{align*}
\eps^2\left\vert A\right\vert\int_{\R}\vert\ueb\px\ueb\vert\vert\pxx\ueb\vert dx= &\int_{\R}\left \vert \eps^{\frac{1}{2}}\sqrt{3}A \ueb\px\ueb \right\vert \left \vert \frac{\eps^{\frac{3}{2}}}{\sqrt{3}}\pxx\ueb \right\vert\\
\le&\frac{3\eps A^2}{2}\norm{\ueb(t,\cdot)\px\ueb(t,\cdot)}^2_{L^2(\R)}+ \frac{\eps^3}{6}\norm{\pxx\ueb(t,\cdot)}^2_{L^2(\R)}.
\end{align*}
Hence, from \eqref{eq:p12},
\begin{equation}
\label{eq:p13}
\begin{split}
&\frac{d}{dt}\left(\frac{K}{4}\norm{\ueb(t,\cdot)}^4_{L^{4}(\R)} +\frac{\eps^2}{2}\norm{\px\ueb(t,\cdot)}^2_{L^2(\R)}\right)\\
&\qquad\quad +\eps\left(3K -\frac{3A^2}{2}\right)\norm{\ueb(t,\cdot)\px\ueb(t,\cdot)}^2_{L^2(\R)} + \frac{5\eps^3}{6}\norm{\pxx\ueb(t,\cdot)}^2_{L^2(\R)}\\
&\qquad\le 3K\beta\int_{\R}\ueb^2\vert\px\ueb\vert \vert\pxx\ueb\vert dx +\frac{7}{3}K\beta\int_{\R}\vert A\ueb^3\vert\vert\px\ueb\vert\vert\pxx\ueb\vert dx.
\end{split}
\end{equation}
Observe that, from \eqref{eq:beta-eps},
\begin{equation}
\label{eq:eps-beta-2}
\beta\le D_1 \eps^4,
\end{equation}
where $D_1$ is a positive constant which will be specified later.
It follows from \eqref{eq:u-l-infty}, \eqref{eq:eps-beta-2} and the Young inequality that
\begin{align*}
&3K\beta\int_{\R}\ueb^2\vert\px\ueb\vert \vert\pxx\ueb\vert dx\le 3K\beta\norm{\ueb(t,\cdot)}^2_{L^{\infty}(\R)}\int_{\R}\vert\px\ueb\vert\vert\pxx\ueb\vert dx\\
&\qquad \le KC_{0}\beta^{\frac{1}{2}}\int_{\R}\vert\px\ueb\vert\vert\pxx\ueb\vert dx \le \int_{\R}\left\vert \sqrt{3}C_{0}D_{1}K\eps^{\frac{1}{2}}\px\ueb\right\vert\left\vert\frac{\eps^{\frac{3}{2}}}{\sqrt{3}}\pxx\ueb\right\vert dx\\
&\qquad \le C_{0}D_{1}^2K^2\eps\norm{\px\ueb(t,\cdot)}^2_{L^2(\R)}+\frac{\eps^3}{6}\norm{\pxx\ueb(t,\cdot)}^2_{L^2(\R)},\\
&\frac{7}{3}K\beta\int_{\R}\vert A\ueb^3\vert\vert\px\ueb\vert\vert\pxx\ueb\vert dx\le \frac{7}{3}K\beta \norm{\ueb(t,\cdot)}^2_{L^{\infty}(\R)}\int_{\R}\vert A\ueb\vert\vert\px\ueb\vert\vert\pxx\ueb\vert dx\\
&\qquad\le  K C_{0}\beta^{\frac{1}{2}}\int_{\R}\vert A\ueb\vert\vert\px\ueb\vert\vert\pxx\ueb\vert dx
\le \int_{\R}\left\vert  \sqrt{3} C_{0}AD_1K\eps^{\frac{1}{2}}\ueb\px\ueb\right\vert \left\vert \frac{\eps^{\frac{3}{2}}}{\sqrt{3}}\pxx\ueb \right\vert dx\\
&\qquad \le C_{0} A^2D_{1}^2 K^2\eps\norm{\ueb(t,\cdot)\px\ueb(t,\cdot)}^2_{L^2(\R)}+ \frac{\eps^3}{6}\norm{\pxx\ueb(t,\cdot)}^2_{L^2(\R)}.
\end{align*}
Therefore, we have
\begin{align*}
&\frac{d}{dt}\left(\frac{K}{4}\norm{\ueb(t,\cdot)}^4_{L^{4}(\R)} +\frac{\eps^2}{2}\norm{\px\ueb(t,\cdot)}^2_{L^2(\R)}\right)\\
&\qquad\quad +\eps\left(3K -\frac{3A^2}{2}\right)\norm{\ueb(t,\cdot)\px\ueb(t,\cdot)}^2_{L^2(\R)} + \frac{5\eps^3}{6}\norm{\pxx\ueb(t,\cdot)}^2_{L^2(\R)}\\
&\qquad \le C_{0}D_{1}^2K^2\eps\norm{\px\ueb(t,\cdot)}^2_{L^2(\R)}+\frac{\eps^3}{3}\norm{\pxx\ueb(t,\cdot)}^2_{L^2(\R)}\\
&\qquad\quad + C_{0}A^2D_{1}^2 K^2\eps\norm{\ueb(t,\cdot)\px\ueb(t,\cdot)}^2_{L^2(\R)},
\end{align*}
that is
\begin{equation}
\label{eq:p18}
\begin{split}
&\frac{d}{dt}\left(\frac{K}{4}\norm{\ueb(t,\cdot)}^4_{L^{4}(\R)} +\frac{\eps^2}{2}\norm{\px\ueb(t,\cdot)}^2_{L^2(\R)}\right)\\
&\qquad\quad +\eps\left(3K -\frac{3A^2}{2}-C_{0}A^2D_{1}^2K^2 \right)\norm{\ueb(t,\cdot)\px\ueb(t,\cdot)}^2_{L^2(\R)}\\
&\qquad\quad + \frac{\eps^3}{2}\norm{\pxx\ueb(t,\cdot)}^2_{L^2(\R)}\le C_{0}D_{1}^2K^2\eps\norm{\px\ueb(t,\cdot)}^2_{L^2(\R)}.
\end{split}
\end{equation}
We search a constant $K$ such that
\begin{equation}
\label{eq:eq-in-k}
C_0A^2D_1^2K^2 -6K +3A^2 <0.
\end{equation}
$K$ does exist if and only if
\begin{equation}
\label{eq:delta}
3-C_0A^4D_1^2 >0.
\end{equation}
Choosing
\begin{equation}
\label{eq:d-1}
D_{1}= \frac{1}{\sqrt{C_0}A^2},
\end{equation}
it follows from \eqref{eq:eq-in-k} and \eqref{eq:d-1} that, there exist $0<K_1<K_2$, such that for every
\begin{equation}
\label{eq:k-3}
K_1<K<K_2
\end{equation}
\eqref{eq:eq-in-k} holds.
Hence, from \eqref{eq:k-3}, choosing $K_1<K_3<K_2$, we get
\begin{equation}
\label{eq:p20}
\begin{split}
&\frac{d}{dt}\left(\frac{K_3}{4}\norm{\ueb(t,\cdot)}^4_{L^{4}(\R)} +\frac{\eps^2}{2}\norm{\px\ueb(t,\cdot)}^2_{L^2(\R)}\right)\\
&\qquad\quad +\eps K_4\norm{\ueb(t,\cdot)\px\ueb(t,\cdot)}^2_{L^2(\R)}+ \frac{\eps^3}{2}\norm{\pxx\ueb(t,\cdot)}^2_{L^2(\R)}\\
&\qquad\le K_5\eps\norm{\px\ueb(t,\cdot)}^2_{L^2(\R)},
\end{split}
\end{equation}
where $K_4$ and $K_5$ are two fixed positive constants.
Integrating \eqref{eq:p20} on $(0,t)$, from \eqref{eq:u0eps} and \eqref{eq:l-2}, we have
\begin{equation}
\label{eq:p21}
\begin{split}
\frac{K_3}{4}\norm{\ueb(t,\cdot)}^4_{L^{4}(\R)} &+\frac{\eps^2}{2}\norm{\px\ueb(t,\cdot)}^2_{L^2(\R)}\\
& +\eps K_4\int_{0}^{t}\norm{\ueb(s,\cdot)\px\ueb(s,\cdot)}^2_{L^2(\R)}ds\\
&+\frac{\eps^3}{2}\int_{0}^{t}\norm{\pxx\ueb(s,\cdot)}^2_{L^2(\R)}ds\\
\le& C_0 + K_5\eps\int_{0}^{t}\norm{\px\ueb(s,\cdot)}^2_{L^2(\R)}ds \le C_0\left(1+K_5\right)\le C_{0}.
\end{split}
\end{equation}
Then,
\begin{equation}
\label{eq:p21}
\begin{split}
\norm{\ueb(t,\cdot)}^4_{L^{4}(\R)}\le & C_0,\\
\eps^2\norm{\px\ueb(t,\cdot)}^2_{L^2(\R)}\le & C_0,\\
\eps \int_{0}^{t}\norm{\ueb(s,\cdot)\px\ueb(s,\cdot)}^2_{L^2(\R)}ds \le & C_0,\\
\eps^3\int_{0}^{t}\norm{\pxx\ueb(s,\cdot)}^2_{L^2(\R)}ds\le &C_0,
\end{split}
\end{equation}
for every $t>0$.
Thanks to \eqref{eq:l-2}, \eqref{eq:eps-beta-2}, \eqref{eq:p21} and the H\"older inequality,
\begin{align*}
&\beta\int_{0}^{t}\norm{\px\ueb(s,\cdot)\pxx\ueb(s,\cdot)}_{L^1(\R)}ds=\frac{\beta}{\eps^2}\int_{0}^{t}\!\!\!\int_{\R}\eps^{\frac{1}{2}}\vert\px\ueb\vert\eps^{\frac{3}{2}}\vert\pxx\vert dsdx\\
&\qquad \le \frac{\beta}{\eps^2}\left(\eps\int_{0}^{t}\norm{\px\ueb(s,\cdot)}^2_{L^2(\R)}ds\right)^{\frac{1}{2}} \left(\eps^3\int_{0}^{t}\norm{\pxx\ueb(s,\cdot)}^2_{L^2(\R)}ds\right)^{\frac{3}{2}}\\
&\qquad \le C_0\frac{\beta}{\eps^2}\le C_{0}D_1\eps^2,
\end{align*}
that is \eqref{eq:ux-uxx}.
Due to \eqref{eq:eps-beta-2} and \eqref{eq:p21},
\begin{align*}
\beta^2\int_{0}^{t}\norm{\pxx\ueb(s,\cdot)}^2_{L^2(\R)}ds= \frac{\beta^2}{\eps^3}\eps^3\int_{0}^{t}\norm{\pxx\ueb(s,\cdot)}^2_{L^2(\R)}ds\le C_0D_1\eps^3,
\end{align*}
which gives \eqref{eq:uxx-l-2}.
It follows from \eqref{eq:u-l-infty}, \eqref{eq:eps-beta-2} and \eqref{eq:p21} that
\begin{align*}
&\beta^2\int_{0}^{t}\norm{\ueb(s,\cdot)\pxx\ueb(s,\cdot)}^2_{L^2(\R)}ds\le \beta^2\norm{\ueb}^2_{L^{\infty}((0,\infty)\times\R)}\int_{0}^{t}\norm{\pxx\ueb(s,\cdot)}^2_{L^2(\R)}ds\\
&\qquad\le \frac{\beta^{\frac{3}{2}}}{\eps^3}\eps^3\int_{0}^{t}\norm{\pxx\ueb(s,\cdot)}^2_{L^2(\R)}ds\le C_0D_1\frac{\eps^6}{\eps^3}\le C_{0}\eps^3,
\end{align*}
that is \eqref{eq:u-uxx-1}.
From \eqref{eq:l-2}, \eqref{eq:u-l-infty}, \eqref{eq:eps-beta-2}, \eqref{eq:p21} and the H\"older inequality,
\begin{align*}
&\beta\int_{0}^t\norm{\ueb(s,\cdot)\px\ueb(s,\cdot)\pxx\ueb(s,\cdot)}_{L^1(\R)} ds\\
&\qquad\le \beta\norm{\ueb}_{L^{\infty}((0,\infty)\times\R)}\int_{0}^t\!\!\!\int_{\R}\vert\px\ueb\vert\vert\pxx\ueb\vert ds dx\\
&\qquad\le  C_{0}\frac{\beta^{\frac{3}{4}}}{\eps^2} \int_{0}^t\!\!\!\int_{\R}\eps^{\frac{1}{2}}\vert\px\ueb\vert\eps^{\frac{3}{2}}\vert\pxx\ueb\vert ds dx\\
&\qquad \le  C_{0}\frac{\beta^{\frac{3}{4}}}{\eps^2}\left(\eps\int_{0}^{t}\norm{\px\ueb(s,\cdot)}^2_{L^2(\R)}ds\right)^{\frac{1}{2}} \left(\eps^3\int_{0}^{t}\norm{\pxx\ueb(s,\cdot)}^2_{L^2(\R)}ds\right)^{\frac{3}{2}}\\
&\qquad \le C_0D_{1}\frac{\eps^3}{\eps^2}\le C_0\eps,
\end{align*}
which gives \eqref{eq:u-ux-uxx}.
\end{proof}

\section{Proof of Theorem \ref{th:main}}
\label{sec:theor}
In this section, we prove Theorem \ref{th:main}. The following technical lemma is needed  \cite{Murat:Hneg}.
\begin{lemma}
\label{lm:1}
Let $\Omega$ be a bounded open subset of $
\R^2$. Suppose that the sequence $\{\mathcal
L_{n}\}_{n\in\mathbb{N}}$ of distributions is bounded in
$W^{-1,\infty}(\Omega)$. Suppose also that
\begin{equation*}
\mathcal L_{n}=\mathcal L_{1,n}+\mathcal L_{2,n},
\end{equation*}
where $\{\mathcal L_{1,n}\}_{n\in\mathbb{N}}$ lies in a
compact subset of $H^{-1}_{loc}(\Omega)$ and
$\{\mathcal L_{2,n}\}_{n\in\mathbb{N}}$ lies in a
bounded subset of $\mathcal{M}_{loc}(\Omega)$. Then $\{\mathcal
L_{n}\}_{n\in\mathbb{N}}$ lies in a compact subset of $H^{-1}_{loc}(\Omega)$.
\end{lemma}
Moreover, we consider the following definition.
\begin{definition}
A pair of functions $(\eta, q)$ is called an  entropy--entropy flux pair if $\eta :\R\to\R$ is a $C^2$ function and $q :\R\to\R$ is defined by
\begin{equation*}
q(u)=\int_{0}^{u} A\xi\eta'(\xi) d\xi.
\end{equation*}
An entropy-entropy flux pair $(\eta,\, q)$ is called  convex/compactly supported if, in addition, $\eta$ is convex/compactly supported.
\end{definition}
Following \cite{LN}, we prove Theorem \ref{th:main}.
\begin{proof}[Proof of Theorem \ref{th:main}]
Let us consider a compactly supported entropy--entropy flux pair $(\eta, q)$. Multiplying \eqref{eq:KS-eps-beta} by $\eta'(\ueb)$, we have
\begin{align*}
\pt\eta(\ueb) + \px q(\ueb) =&\eps \eta'(\ueb) \pxx\ueb - \beta \eta'(\ueb) \pxxx\ueb\\
&-B\beta\eta'(\ueb)\px\left(\ueb\pxx\ueb\right)-C\beta\eta'(\ueb)\px\ueb\pxx\ueb\\
=& I_{1,\,\eps,\,\beta}+I_{2,\,\eps,\,\beta}+ I_{3,\,\eps,\,\beta} + I_{4,\,\eps,\,\beta}+I_{5,\,\eps,\,\beta}+I_{6,\,\eps,\,\beta}+I_{7,\,\eps,\,\beta},
\end{align*}
where
\begin{equation}
\begin{split}
\label{eq:12000}
I_{1,\,\eps,\,\beta}&=\px(\eps\eta'(\ueb)\px\ueb),\\
I_{2,\,\eps,\,\beta}&= -\eps\eta''(\ueb)(\px\ueb)^2,\\
I_{3,\,\eps,\,\beta}&= \px(-\beta\eta'(\ueb)\pxx\ueb),\\
I_{4,\,\eps,\,\beta}&= \beta\eta''(\ueb)\px\ueb\pxx\ueb,\\
I_{5,\,\eps,\,\beta}&= \px\left(-B\beta  \eta'(\ueb)\ueb\pxx\ueb\right),\\
I_{6,\,\eps,\,\beta}&= B\beta  \eta''(\ueb)\ueb\px\ueb\pxx\ueb,\\
I_{7,\,\eps,\,\beta}&=-C \beta\eta'(\ueb)\px\ueb\pxx\ueb.
\end{split}
\end{equation}
We have
\begin{equation*}
\label{eq:H1}
I_{1,\,\eps,\,\beta}\to0 \quad \text{in $H^{-1}((0,T) \times\R),\,T>0$, as $\eps\to 0$.}
\end{equation*}
Indeed, thanks to Lemma \ref{lm:l-2},
\begin{align*}
\norm{\eps\eta'(\ueb)\px\ueb}^2_{L^2((0,T)\times\R))}&\leq \norm{\eta'}^2_{L^{\infty}(\R)}\eps ^2\int_{0}^{T} \norm{\px\ueb(s,\cdot)}^2_{L^2(\R)}ds\\
&\leq \norm{\eta'}^2_{L^{\infty}(\R)}\eps C_{0} \to 0.
\end{align*}
We claim that
\begin{equation*}
\{I_{2,\,\eps,\,\beta}\}_{\eps,\,\beta >0} \quad\text{is bounded in $L^1((0,T)\times\R),\, T>0$}.
\end{equation*}
Again by Lemma \ref{lm:l-2},
\begin{align*}
\norm{ \eps\eta''(\ueb)(\px\ueb)^2}_{L^1((0,T)\times\R)}& \leq \norm{\eta''}_{L^\infty(\R)}\eps\int_{0}^{T}\norm{\px\ueb(s,\cdot)}^2_{L^2(\R)}ds\\
&\leq\norm{\eta''}_{L^\infty (\R)}C_0.
\end{align*}
We have that
\begin{equation*}
I_{3,\,\eps,\,\beta}\to0 \quad \text{in $H^{-1}((0,T) \times\R),\,T>0,$ as $\eps\to 0$.}
\end{equation*}
Thanks to Lemma \ref{lm:ux-l-2},
\begin{align*}
\norm{\beta^2\eta'(\ueb)\pxx\ueb}^2_{L^2((0,T)\times\R))}&\leq \norm{\eta'}^2_{L^{\infty}(\R)}\beta ^2\int_{0}^{T} \norm{\pxx\ueb(s,\cdot)}^2_{L^2(\R)}ds\\
&\leq \norm{\eta'}^2_{L^{\infty}(\R)}C(T)\eps \to 0.
\end{align*}
We  show that
\begin{equation*}
I_{4,\,\eps,\,\beta}\to0 \quad \text{in $L^1((0,T) \times\R),\,T>0,$ as $\eps\to 0$.}
\end{equation*}
Again by Lemma \ref{lm:ux-l-2},
\begin{align*}
&\norm{\beta\eta''(\ueb)\px\ueb\pxx\ueb}_{L^1((0,T)\times\R)}\\
&\qquad\leq \norm{\eta''}_{L^{\infty}(\R)}\beta\int_{0}^{T} \norm{\px\ueb(s,\cdot)\pxx\ueb(s,\cdot)}_{L^1(\R)}ds\\
&\qquad\leq \norm{\eta''}_{L^{\infty}(\R)}C_0\eps^2\to 0.
\end{align*}
We claim that
\begin{equation*}
I_{5,\,\eps,\,\beta}\to0 \quad \text{in $H^{-1}((0,T) \times\R),\,T>0,$ as $\eps\to 0$.}
\end{equation*}
By Lemma  \ref{lm:ux-l-2},
\begin{align*}
&\norm{B\beta  \eta'(\ueb)\ueb\pxx\ueb}^2_{L^2((0,T)\times\R)}\\ &\qquad \le B^2\beta^2\norm{\eta'}^2_{L^{\infty}(\R)}\beta^2\int_{0}^{T}\norm{\ueb(s,\cdot)\pxx\ueb(s,\cdot)}^2_{L^2(\R)}ds\\
&\qquad\le B^2\norm{\eta'}^2_{L^{\infty}(\R)}C_0\eps^3\to 0.
\end{align*}
We have that
\begin{equation*}
I_{6,\,\eps,\,\beta}\to0 \quad \text{in $L^1((0,T) \times\R),\,T>0,$ as $\eps\to 0$.}
\end{equation*}
Again by Lemma \ref{lm:ux-l-2},
\begin{align*}
&\norm{B\beta\eta''(\ueb)\ueb\px\ueb\pxx\ueb}_{L^1((0,T)\times\R)}\\
&\qquad\le \left\vert B\right\vert\norm{\eta''}_{L^{\infty}(\R)}\beta\int_{0}^{T}\norm{\ueb(s,\cdot)\px\ueb(s,\cdot)\pxx\ueb(s,\cdot)}_{L^1(\R)} ds\\
&\qquad\le \left\vert B\right\vert\norm{\eta''}_{L^{\infty}(\R)} C_0\eps \to 0.
\end{align*}
We claim that
\begin{equation*}
I_{7,\,\eps,\,\beta}\to0 \quad \text{in $L^1((0,T) \times\R),\,T>0,$ as $\eps\to 0$.}
\end{equation*}
By Lemma \ref{lm:ux-l-2},
\begin{align*}
&\norm{C \beta\eta'(\ueb)\px\ueb\pxx\ueb}_{L^1((0,T)\times\R)}\\
&\qquad \le \left\vert C\right\vert\norm{\eta'}_{L^{\infty}(\R)}\beta\int_{0}^{T}\norm{\px\ueb(s,\cdot)\pxx\ueb(s,\cdot)}_{L^1(\R)}ds\\
&\qquad\le \left\vert C\right \vert\norm{\eta'}_{L^{\infty}(\R)}C_0\eps^2\to 0.
\end{align*}
Therefore, \eqref{eq:con-u} follows from Lemma \ref{lm:1} and the $L^p$ compensated compactness  \cite{SC}.

We  have to show that \eqref{eq:entropy1} holds. We begin by proving that $u$ is a distributional solution of \eqref{eq:BU}.
Let $ \phi\in C^{\infty}(\R^2)$ be a test function with  compact support. We have to prove that
\begin{equation}
\label{eq:k1}
\int_{0}^{\infty}\!\!\!\!\!\int_{\R}\left(u\pt\phi+\frac{Au^2}{2}\px\phi\right)dtdx +\int_{\R}u_{0}(x)\phi(0,x)dx=0.
\end{equation}
We have that
\begin{equation}
\label{eq:58}
\begin{split}
\int_{0}^{\infty}\!\!&\!\!\!\int_{\R}\left(u_{\eps_{n}, \beta_{n}}\pt\phi+\frac{Au^2_{\eps_n, \beta_{n}}}{2}\px\phi\right)dtdx +\int_{\R}u_{0,\eps_n,\beta_n}(x)\phi(0,x)dx\\
&+\eps_{n}\int_{0}^{\infty}\!\!\!\!\!\int_{\R}u_{\eps_{n},\beta_{n}}\pxx\phi dtdx + \eps_n\int_{0}^{\infty}u_{0,\eps_{n},\beta_{n}}(x)\pxx\phi(0,x)dx\\
&+ \beta_n\int_{0}^{\infty}\!\!\!\!\int_{\R}u_{\eps_n,\beta_n}\pxxx\phi dt dx + \beta_n\int_{0}^{\infty}u_{0,\eps_n,\beta_n}(x)\pxxx\phi(0,x)dx\\
=&B\beta_{n}\int_{0}^{\infty}\!\!\!\!\int_{\R}u_{\eps_n,\beta_n}\pxx u_{\eps_n,\beta_n}\px\phi dtdx -C\beta_{n}\int_{0}^{\infty}\!\!\!\!\int_{\R}\px u_{\eps_n,\beta_n}\pxx u_{\eps_n,\beta_n}\phi dtdx.
\end{split}
\end{equation}
Let us show that
\begin{equation}
\label{eq:54}
B\beta_{n}\int_{0}^{\infty}\!\!\!\!\int_{\R}u_{\eps_n,\beta_n}\pxx u_{\eps_n,\beta_n}\px\phi dtdx \to 0.
\end{equation}
Fix $T>0$. Due to \eqref{eq:beta-eps}, \eqref{eq:u-l-infty}, Lemma \ref{lm:ux-l-2} and the H\"older inequality,
\begin{align*}
&\left \vert B\right \vert \beta_{n} \left \vert \int_{0}^{\infty}\!\!\!\!\int_{\R}u_{\eps_n,\beta_n}\pxx u_{\eps_n,\beta_n}\px\phi dtdx\right\vert \\
&\qquad \le \left \vert B\right \vert \beta_{n}\int_{0}^{\infty}\!\!\!\!\int_{\R}\vert u_{\eps_n,\beta_n} \vert \vert  \pxx u_{\eps_n,\beta_n}\vert\vert\px\phi \vert dtdx\\
&\qquad\le \left \vert B\right \vert\beta_{n}\norm{ u_{\eps_n,\beta_n}}_{L^{\infty}(0,\infty)\times\R)}\int_{0}^{\infty}\!\!\!\!\int_{\R}\vert  \pxx u_{\eps_n,\beta_n}\vert\vert\px\phi \vert dtdx\\
&\qquad \le \left \vert B\right \vert C_{0}\beta^{\frac{3}{4}}_{n}\norm{\pxx u_{\eps_n,\beta_n}}_{L^2(\supp(\px\phi))}\norm{\px\phi}_{L^2(\supp(\px\phi))}\\
&\qquad \le  \left \vert B\right \vert C_{0} \eps^{3}_{n} \norm{\pxx u_{\eps_n,\beta_n}}_{L^2((0,T)\times\R)}\norm{\px\phi}_{L^2((0,T)\times\R)}\\
&\qquad \le \left \vert B\right \vert C_{0}\eps^{\frac{3}{2}}_{n}\to 0,
\end{align*}
that is \eqref{eq:54}.

We prove that
\begin{equation}
\label{eq:55}
-C\beta_{n}\int_{0}^{\infty}\!\!\!\!\int_{\R}\px u_{\eps_n,\beta_n}\pxx u_{\eps_n,\beta_n}\phi dtdx\to 0.
\end{equation}
Fix $T>0$. Thanks to Lemma \ref{lm:ux-l-2},
\begin{align*}
&\left \vert C\right\vert \beta_{n}\left\vert \int_{0}^{\infty}\!\!\!\!\int_{\R}\px u_{\eps_n,\beta_n}\pxx u_{\eps_n,\beta_n}\phi dtdx \right\vert\\
&\qquad \le \left\vert C\right\vert \beta_{n}\int_{0}^{\infty}\!\!\!\!\int_{\R}\vert\px u_{\eps_n,\beta_n}\pxx u_{\eps_n,\beta_n}\vert \vert \phi\vert dtdx\\
&\qquad \le \left\vert C\right\vert \norm{\phi}_{L^{\infty}(\supp(\phi))}\beta_{n}\norm{\px u_{\eps_n,\beta_n}\pxx u_{\eps_n,\beta_n}}_{L^{1}(\supp(\phi))}\\
&\qquad \le \left\vert C\right\vert \norm{\phi}_{L^{\infty}((0,T)\times\R)}\beta_{n}\norm{\px u_{\eps_n,\beta_n}\pxx u_{\eps_n,\beta_n}}_{L^{1}((0,T)\times\R)}\\
&\qquad \le  \left\vert C\right\vert \norm{\phi}_{L^{\infty}((0,T)\times\R)}C_{0}\eps^2\to 0,
\end{align*}
which gives \eqref{eq:55}.
Therefore, \eqref{eq:k1} follows from \eqref{eq:u0eps}, \eqref{eq:con-u}, \eqref{eq:58}, \eqref{eq:54} and \eqref{eq:55}.

We conclude by proving that $u$ is the unique entropy solution of \eqref{eq:BU}. Fix $T>0$. Let us consider a compactly supported entropy--entropy flux pair $(\eta, q)$, and $\phi\in C^{\infty}_{c}((0,\infty)\times\R)$ a non--negative function. We have to prove that
\begin{equation}
\label{eq:u-entropy-solution}
\int_{0}^{\infty}\!\!\!\!\!\int_{\R}(\pt\eta(u)+ \px q(u))\phi dtdx\le0.
\end{equation}
We have
\begin{align*}
&\int_{0}^{\infty}\!\!\!\!\!\int_{\R}(\px\eta(u_{\eps_{n},\,\beta_{n}})+\px q(u_{\eps_{n},\,\beta_{n}}))\phi dtdx\\
&\qquad=\eps_{n}\int_{0}^{\infty}\!\!\!\!\!\int_{\R}\px(\eta'(u_{\eps_{n},\,\beta_{n}})\px u_{\eps_{n},\,\beta_{n}})\phi dtdx\\ &\qquad\quad-\eps_{n}\int_{0}^{\infty}\!\!\!\!\!\int_{\R} \eta''(u_{\eps_{n},\,\beta_{n}})(\px u_{\eps_{n},\,\beta_{n}})^2\phi dtdx\\
&\qquad\quad -\beta_{n}\int_{0}^{\infty}\!\!\!\!\!\int_{\R}\px(\eta'(u_{\eps_{n},\,\beta_{n}})\pxx u_{\eps_{n},\,\beta_{n}})\phi dtdx\\
&\qquad\quad+\beta_{n}\int_{0}^{\infty}\!\!\!\!\!\int_{\R}\eta''(u_{\eps_{n},\,\beta_{n}})\px u_{\eps_{n},\,\beta_{n}}\pxx u_{\eps_{n},\,\beta_{n}}\phi dtdx\\
&\qquad\quad -B\beta_{n}\int_{0}^{\infty}\!\!\!\!\!\int_{\R}\px(\eta'(u_{\eps_{n},\,\beta_{n}})u_{\eps_{n},\,\beta_{n}} \pxx u_{\eps_{n},\,\beta_{n}})\phi dtdx\\
&\qquad\quad +B\beta_{n}\int_{0}^{\infty}\!\!\!\!\!\int_{\R}\eta''(u_{\eps_{n},\,\beta_{n}})u_{\eps_{n},\,\beta_{n}}\px u_{\eps_{n},\,\beta_{n}} \pxx u_{\eps_{n},\,\beta_{n}}\phi dtdx\\
&\qquad -C\beta_{n}\int_{0}^{\infty}\!\!\!\!\!\int_{\R}\eta'(u_{\eps_{n},\,\beta_{n}})\px u_{\eps_{n},\,\beta_{n}}\pxx u_{\eps_{n},\,\beta_{n}}\phi dtdx\\
&\qquad\le - \eps_{n}\int_{0}^{\infty}\!\!\!\!\!\int_{\R}\eta'(u_{\eps_{n},\,\beta_{n}})\px u_{\eps_{n},\,\beta_{n}}\px\phi dtdx\\
&\qquad\quad +\beta_{n}\int_{0}^{\infty}\!\!\!\!\!\int_{\R}\eta'(u_{\eps_{n},\,\beta_{n}})\pxx u_{\eps_{n},\,\beta_{n}}\px\phi dtdx\\
&\qquad\quad+\beta_{n}\int_{0}^{\infty}\!\!\!\!\!\int_{\R}\eta''(u_{\eps_{n},\,\beta_{n}})\px u_{\eps_{n},\,\beta_{n}}\pxx u_{\eps_{n},\,\beta_{n}}\phi dtdx\\
&\qquad\quad +B\beta_{n}\int_{0}^{\infty}\!\!\!\!\!\int_{\R}\eta'(u_{\eps_{n},\,\beta_{n}})u_{\eps_{n},\,\beta_{n}} \pxx u_{\eps_{n},\,\beta_{n}}\px\phi dtdx\\
&\qquad\quad +B\beta_{n}\int_{0}^{\infty}\!\!\!\!\!\int_{\R}\eta''(u_{\eps_{n},\,\beta_{n}})u_{\eps_{n},\,\beta_{n}}\px u_{\eps_{n},\,\beta_{n}} \pxx u_{\eps_{n},\,\beta_{n}}\phi dtdx\\
&\qquad -C\beta_{n}\int_{0}^{\infty}\!\!\!\!\!\int_{\R}\eta'(u_{\eps_{n},\,\beta_{n}})\px u_{\eps_{n},\,\beta_{n}}\pxx u_{\eps_{n},\,\beta_{n}}\phi dtdx\\
&\qquad \le \eps_{n}\int_{0}^{\infty}\!\!\!\!\!\int_{\R}\vert\eta'(u_{\eps_{n},\,\beta_{n}})\vert\vert\px u_{\eps_{n},\,\beta_{n}}\vert\vert\px\phi\vert dtdx\\
&\qquad\quad +\beta_{n}\int_{0}^{\infty}\!\!\!\!\!\int_{\R}\vert\eta'(u_{\eps_{n},\,\beta_{n}})\vert\vert\pxx u_{\eps_{n},\,\beta_{n}}\vert\vert\px\phi\vert d tdx\\
&\qquad\quad +\beta_{n}\int_{0}^{\infty}\!\!\!\!\!\int_{\R}\vert\eta''(u_{\eps_{n},\,\beta_{n}})\vert\vert\px u_{\eps_{n},\,\beta_{n}}\pxx u_{\eps_{n},\,\beta_{n}}\vert\vert\phi\vert dtdx\\
&\qquad\quad +\left\vert B\right\vert\beta_{n}\int_{0}^{\infty}\!\!\!\!\!\int_{\R}\vert\eta'(u_{\eps_{n},\,\beta_{n}})\vert\vert u_{\eps_{n},\,\beta_{n}}\pxx u_{\eps_{n},\,\beta_{n}}\vert\vert\px\phi\vert dtdx\\
&\qquad\quad+\left\vert B\right\vert\beta_{n}\int_{0}^{\infty}\!\!\!\!\!\int_{\R}\vert\eta''(u_{\eps_{n},\,\beta_{n}})\vert\vert u_{\eps_{n},\,\beta_{n}}\px u_{\eps_{n},\,\beta_{n}}\pxx u_{\eps_{n},\,\beta_{n}}\vert\vert\phi\vert dtdx\\
&\qquad\quad +\left\vert C\right\vert\beta_{n}\int_{0}^{\infty}\!\!\!\!\!\int_{\R}\vert\eta'(u_{\eps_{n},\,\beta_{n}})\vert\vert\px u_{\eps_{n},\,\beta_{n}}\pxx u_{\eps_{n},\,\beta_{n}}\vert\vert\phi\vert dtdx.
\end{align*}
Hence, from \eqref{eq:u-l-infty},
\begin{align*}
&\int_{0}^{\infty}\!\!\!\!\!\int_{\R}(\px\eta(u_{\eps_{n},\,\beta_{n}})+\px q(u_{\eps_{n},\,\beta_{n}}))\phi dtdx\\
&\qquad\le  \eps_{n} \norm{\eta'}_{L^{\infty}(\R)}\norm{\px u_{\eps_{n},\,\beta_{n}}}_{L^2(\supp(\px\phi))}\norm{\px\phi}_{L^2(\supp(\px\phi))}\\
&\qquad\quad+ \beta_{n} \norm{\eta'}_{L^{\infty}(\R)}\norm{\pxx u_{\eps_{n},\,\beta_{n}}}_{L^2(\supp(\px\phi))}\norm{\px\phi}_{L^2(\supp(\px\phi))}\\
&\qquad\quad +\beta_{n} \norm{\eta''}_{L^{\infty}(\R)}\norm{\phi}_{L^{\infty}(\R)}\norm{\px u_{\eps_{n},\,\beta_{n}}\pxx u_{\eps_{n},\,\beta_{n}}}_{L^1(\supp(\px\phi))}\\
&\qquad\quad +\left \vert B\right\vert\norm{\eta'}_{L^{\infty}(\R)} \beta_n\norm{u_{\eps_{n},\,\beta_{n}}}_{L^{\infty}((0,\infty)\times\R)}\int_{0}^{\infty}\!\!\!\!\!\int_{\R}\vert\pxx u_{\eps_{n},\,\beta_{n}}\vert\vert\px\phi\vert dtdx\\
&\qquad\quad +\left \vert B\right\vert\norm{\eta''}_{L^{\infty}(\R)} \beta_n\norm{u_{\eps_{n},\,\beta_{n}}}_{L^{\infty}((0,\infty)\times\R)}\int_{0}^{\infty}\!\!\!\!\!\int_{\R}\vert \px u_{\eps_{n},\,\beta_{n}}\vert\vert\pxx u_{\eps_{n},\,\beta_{n}}\vert\vert\px\phi\vert dtdx\\
&\qquad\quad +\beta_{n}\left \vert C\right\vert\norm{\eta'}_{L^{\infty}(\R)}\norm{\phi}_{L^{\infty}(\R)}\norm{\px u_{\eps_{n},\,\beta_{n}}\pxx u_{\eps_{n},\,\beta_{n}}}_{L^1(\supp(\px\phi))}\\
&\qquad\le  \eps_{n} \norm{\eta'}_{L^{\infty}(\R)}\norm{\px u_{\eps_{n},\,\beta_{n}}}_{L^2((0,T)\times\R)}\norm{\px\phi}_{L^2((0,T)\times\R)}\\
&\qquad\quad+ \beta_{n} \norm{\eta'}_{L^{\infty}(\R)}\norm{\pxx u_{\eps_{n},\,\beta_{n}}}_{L^2((0,T)\times\R)}\norm{\px\phi}_{L^2((0,T)\times\R)}\\
&\qquad\quad+\beta_{n} \norm{\eta''}_{L^{\infty}(\R)}\norm{\phi}_{L^{\infty}(\R^{+}\times\R)}\norm{\px u_{\eps_{n},\,\beta_{n}}\pxx u_{\eps_{n},\,\beta_{n}}}_{L^1((0,T)\times\R)}\\
&\qquad\quad+\beta_{n}\left \vert C\right\vert\norm{\eta'}_{L^{\infty}(\R)}\beta_{n}\norm{\phi}_{L^{\infty}(\R^{+}\times\R)}\norm{\px u_{\eps_{n},\,\beta_{n}}\pxx u_{\eps_{n},\,\beta_{n}}}_{L^1((0,T)\times\R)}\\
&\qquad\quad + C_0\left\vert B\right\vert\norm{\eta''}_{L^{\infty}(\R)}\beta_{n}^{\frac{3}{4}}\int_{0}^{\infty}\!\!\!\!\!\int_{\R}\vert\pxx u_{\eps_{n},\,\beta_{n}}\vert\vert\px\phi\vert dtdx\\
&\qquad\quad +C_{0}\left \vert B\right\vert\norm{\eta''}_{L^{\infty}(\R)}\beta_{n}^{\frac{3}{4}}\int_{0}^{\infty}\!\!\!\!\!\int_{\R}\vert \px u_{\eps_{n},\,\beta_{n}}\vert\vert\pxx u_{\eps_{n},\,\beta_{n}}\vert\vert\px\phi\vert dtdx,
\end{align*}
that is
\begin{equation}
\label{eq:70}
\begin{split}
&\int_{0}^{\infty}\!\!\!\!\!\int_{\R}(\px\eta(u_{\eps_{n},\,\beta_{n}})+\px q(u_{\eps_{n},\,\beta_{n}}))\phi dtdx\\
&\qquad\quad\le  C_1 \eps_{n}\norm{\px u_{\eps_{n},\,\beta_{n}}}_{L^2((0,T)\times\R)}+ C_1\beta_{n}\norm{\pxx u_{\eps_{n},\,\beta_{n}}}_{L^2((0,T)\times\R)}\\
&\qquad\quad +C_1\beta_{n}\norm{\px u_{\eps_{n},\,\beta_{n}}\pxx u_{\eps_{n},\,\beta_{n}}}_{L^1((0,T)\times\R)} +C_1\beta_{n}^{\frac{3}{4}}\int_{0}^{\infty}\!\!\!\!\!\int_{\R}\vert\pxx u_{\eps_{n},\,\beta_{n}}\vert\vert\px\phi\vert dtdx\\
&\qquad\quad+ C_1\beta_{n}^{\frac{3}{4}}\int_{0}^{\infty}\!\!\!\!\!\int_{\R}\vert \px u_{\eps_{n},\,\beta_{n}}\vert\vert\pxx u_{\eps_{n},\,\beta_{n}}\vert\vert\px\phi\vert dtdx,
\end{split}
\end{equation}
where $C_1$ is a suitable positive constant.

Let us show that
\begin{equation}
\label{eq:60}
\beta_{n}^{\frac{3}{4}}\int_{0}^{\infty}\!\!\!\!\!\int_{\R}\vert\pxx u_{\eps_{n},\,\beta_{n}}\vert\vert\px\phi\vert dtdx\to 0.
\end{equation}
Due \eqref{eq:beta-eps}, Lemma \eqref{lm:ux-l-2} and the H\"older inequality,
\begin{align*}
&\beta_{n}^{\frac{3}{4}}\int_{0}^{\infty}\!\!\!\!\!\int_{\R}\vert\pxx u_{\eps_{n},\,\beta_{n}}\vert\vert\px\phi\vert dtdx\\
&\qquad \le C_0 \eps_{n}^3 \norm{\pxx u_{\eps_{n},\,\beta_{n}}}_{L^2(\supp(\px\phi))}\norm{\px\phi}_{L^2(\supp(\px\phi))}\\
&\qquad \le C_{0}\eps_{n}^3 \norm{\pxx u_{\eps_{n},\,\beta_{n}}}_{L^2((0,T)\times\R)} \norm{\px\phi}_{L^2((0,T)\times\R)}\\
&\qquad \le C_{0}\eps_{n}^{\frac{3}{2}} \norm{\px\phi}_{L^2((0,T)\times\R)}\to 0,
\end{align*}
that is \eqref{eq:60}.

We claim that
\begin{equation}
\label{eq:61}
\beta_{n}^{\frac{3}{4}}\int_{0}^{\infty}\!\!\!\!\!\int_{\R}\vert \px u_{\eps_{n},\,\beta_{n}}\vert\vert\pxx u_{\eps_{n},\,\beta_{n}}\vert\vert\px\phi\vert dtdx\to0.
\end{equation}
Thanks to Lemmas \ref{lm:l-2}, \ref{lm:ux-l-2} and the H\"older inequality,
\begin{align*}
&\beta_{n}^{\frac{3}{4}}\int_{0}^{\infty}\!\!\!\!\!\int_{\R}\vert \px u_{\eps_{n},\,\beta_{n}}\vert\vert\pxx u_{\eps_{n},\,\beta_{n}}\vert\vert\px\phi\vert dtdx\\
&\qquad\quad \le C_{0}\eps_{n}^3 \norm{\phi}_{L^{\infty}(\R^{+}\times\R)}\norm{\px u_{\eps_{n},\,\beta_{n}}\pxx u_{\eps_{n},\,\beta_{n}}}_{L^1(\supp(\phi))}\\
&\qquad\quad\le  C_{0} \norm{\phi}_{L^{\infty}(\R^{+}\times\R)}\eps_{n}\int_{0}^{T}\!\!\!\int_{\R}\eps_{n}^{\frac{1}{2}}\vert \px u_{\eps_{n},\,\beta_{n}}\vert \eps_{n}^{\frac{3}{2}}\vert\pxx u_{\eps_{n},\,\beta_{n}}\vert dt dx\\
&\qquad\quad \le C_{0} \norm{\phi}_{L^{\infty}(\R^{+}\times\R)}\eps_{n}\left(\eps_{n}\int_{0}^{T}\norm{\px u_{\eps_{n},\,\beta_{n}}(t,\cdot)}_{L^2(\R)}dt\right)^{\frac{1}{2}}\\
&\qquad\qquad\cdot\left(\eps_{n}^3\int_{0}^{T}\norm{\pxx u_{\eps_{n},\,\beta_{n}}(t,\cdot)}_{L^2(\R)}dt\right)^{\frac{1}{2}}\\
&\qquad\quad \le C_{0}\norm{\phi}_{L^{\infty}(\R^{+}\times\R)}\eps\to 0,
\end{align*}
which gives \eqref{eq:60}.

Finally, \eqref{eq:u-entropy-solution} follows from \eqref{eq:beta-eps}, \eqref{eq:con-u}, \eqref{eq:70}, \eqref{eq:60}, \eqref{eq:61} and Lemmas \ref{lm:l-2} and \ref{lm:ux-l-2}.
\end{proof}

\appendix
\label{appen1}\section{On the case $A=\left(C+\alpha\right)^{2n}$}
Theorem \ref{th:main} holds also in the cases $ A=C^2$ and $ A=C^{2n}$, with $ C\neq 0$ and $n\in\N$. Indeed, from \eqref{eq:syst1}, if $ A=C^2$, we get $ C=-3$, while if $ A=C^{2n}$, we obtain $C=-3^{\frac{1}{2n-1}}$.
If $A=C^{2n+1}$, from \eqref{eq:syst1}, we get
\begin{equation*}
C^{2n}+3=0,
\end{equation*}
which does not have solutions in $\R$.

In this section, we prove that Theorem \ref{th:main} holds also in the case $A= \left (C+\alpha\right)^{2n}$, where $\alpha$ is a suitable real number.
We only need to prove the following result
\begin{lemma}
Assume that
\begin{equation}
\label{eq:A1}
A= \left (C+\alpha\right)^n.
\end{equation}
If
\begin{equation}
\label{eq:alpha-1}
\alpha \le 3^{\frac{1}{2n-1}}\left(\frac{1}{2n}\right)^{\frac{2n}{2n-1}}+\left(\frac{3}{2n}\right)^{\frac{1}{2n-1}},
\end{equation}
then \eqref{eq:l-2} holds.
\end{lemma}
\begin{proof}
We begin by observing that, by \eqref{eq:syst1}, we have
\begin{equation*}
\left(C+\alpha\right)^{2n} +3C=0,
\end{equation*}
that is
\begin{equation}
\label{eq:eq-in-C}
\left(C+\alpha\right)^{2n} +3\left(C+\alpha\right) -3\alpha =0.
\end{equation}
Let us consider the following function
\begin{equation}
\label{eq:def-di-g}
g(X)= X^{2n}+3X -3\alpha.
\end{equation}
We observe that
\begin{equation}
\label{eq:teozeri}
\lim_{X\to -\infty} g(X)=\infty, \quad \lim_{X\to \infty} g(X)=\infty.
\end{equation}
Since $g'(X)=2n X^{2n-1} +3$, we have that
\begin{equation}
\label{eq:g-cres}
g\quad \textrm{is increasing in}\quad \left(-\left(\frac{3}{2n}\right)^{\frac{1}{2n-1}}, \infty\right).
\end{equation}
From \eqref{eq:alpha-1},
\begin{equation}
\label{eq:g-in-x0}
g(X_0) \le 0, \quad X_{0} = -\left(\frac{3}{2n}\right)^{\frac{1}{2n-1}}.
\end{equation}
Then, it follows from \eqref{eq:teozeri}, \eqref{eq:g-cres} and \eqref{eq:g-in-x0} that the function $g$ has only two zeros $X_1<0<X_2$.
Hence, from \eqref{eq:A1},
\begin{equation*}
\textrm{$A=X_1^{2n}$, or $A=X_2^{2n}$.}
\end{equation*}
Therefore, arguing as in Lemma \ref{lm:l-2}, we have \eqref{eq:l-2}.
\end{proof}

\end{document}